\documentclass{article}
\usepackage{fullpage,amsmath,amssymb}
\usepackage{hyperref}
\usepackage{calc}  
\usepackage{enumitem}  
\usepackage{fancyhdr}  
\usepackage[margin=1in]{geometry}  
\usepackage{lastpage}  
\usepackage{relsize}  
\usepackage[normalem]{ulem}  
\usepackage{url}  
\usepackage{xfrac}  
\usepackage{amsmath}
\usepackage{amsthm}
\usepackage{listings}
\usepackage{algpseudocode}
\usepackage{algorithm}
\newcommand{\mC}{\mathcal{C}}
\newcommand{\mB}{\mathcal{B}}
\newcommand{\mP}{\mathcal{P}}

\newtheorem{theorem}{Theorem}[section]
\newtheorem{lemma}[theorem]{Lemma}
\newtheorem{claim}[theorem]{Claim}
\newtheorem{definition}[theorem]{Definition}

\newtheorem{problem}[theorem]{Problem}
\newtheorem{proposition}[theorem]{Proposition}

\title{Disperse Hypergraphs}
\author{Lior Gishboliner\thanks{Department of Mathematics, University of Toronto, Canada.
\emph{Email}: \href{mailto:lior.gishboliner@utoronto.ca}{\tt lior.gishboliner@utoronto.ca}. Research supported by the NSERC Discovery Grant ``Problems in Extremal and Probabilistic Combinatorics".
}
\and Ethan Honest\thanks{Department of Mathematics, University of Toronto, Canada.
\emph{Email}: \href{mailto:lior.gishboliner@utoronto.ca}{\tt ethan.honest@mail.utoronto.ca}.}
}

\hypersetup{pdfpagemode=Fullscreen,
  colorlinks=true,
  linkfileprefix={}}

\begin{document}

\maketitle

\begin{abstract}
    For $\ell \geq 3$, an $\ell$-uniform hypergraph is {\em disperse} if the number of edges induced by any set of $\ell+1$ vertices is 0, 1, $\ell$ or $\ell+1$. We show that every disperse $\ell$-uniform hypergraph on $n$ vertices contains a clique or independent set of size $n^{\Omega_{\ell}(1)}$, answering a question of the first author and Tomon. To this end, we prove several structural properties of disperse hypergraphs. 
\end{abstract}

\section{Introduction}

The Erd\H{o}s-Hajnal conjecture \cite{EH_conjecture} is a fundamental problem in extremal graph theory, stating that for every fixed graph $H$, every induced $H$-free $n$-vertex graph has a homogeneous set (i.e., a clique or independent set) of size at least $n^{c_H}$, where $c_H > 0$ is a constant depending only on $H$. This is in sharp contrast to general graphs, which may only have homogeneous sets of size $O(\log n)$. Despite significant recent progress (see, e.g., \cite{BNSS,NSS_VC,NSS_P5} and the references therein), the Erd\H{o}s-Hajnal conjecture remains open in general.

It is natural to ask for analogues of the Erd\H{o}s-Hajnal conjecture for hypergraphs. We will use the term {\em $\ell$-graph} as shorthand for $\ell$-uniform hypergraph. 
Let $\log_k(x)$ denote the $k$-times iterated logarithm, i.e., $\log_0(x) = x$, $\log_1(x) = \log(x)$ and $\log_{k}(x) = \log(\log_{k-1}(x))$. 
Let $h_\ell(n)$ denote the maximum size of a homogeneous set guaranteed to exist in every $\ell$-vertex $n$-graph. 
It is well-known \cite{Erdos_Rado} that $h_{\ell}(n) \geq \left( \log_{\ell-1}(n)\right)^{\Omega(1)}$, and it is conjectured that the number $\ell-1$ of iterated logarithms is best possible, i.e.,
$h_{\ell}(n) \leq \left( \log_{\ell-1}(n)\right)^{O(1)}$.
Thus, a natural analogue of the Erd\H{o}s-Hajnal conjecture for $\ell$-graphs would be that for every fixed $\ell$-graph $H$, every induced $H$-free $n$-vertex $\ell$-graph has a homogeneous set of size 
$\left( \log_{\ell-2}(n) \right)^{c_H}$. 
However, there is some evidence that this might not be true for every $H$. Indeed, it is known that for uniformity $\ell \geq 4$, the so-called {\em stepping up construction} gives a tight lower bound for Ramsey numbers in terms of the number of iterated logarithms (though this number is not known, as for the critical case $\ell=3$ we only know that the number of logarithms is at least 1 and at most 2; this corresponds to at least $\ell-2$ and at most $\ell-1$ logarithms for uniformity $\ell$). Conlon, Fox and Sudakov \cite{CFS_EH} showed that the stepping up construction avoids certain $H$ as induced subgraphs. Thus, for $\ell \geq 4$, there exist $\ell$-graphs $H$ such that induced $H$-free $\ell$-graphs do not have significantly larger homogeneous sets than general $\ell$-graphs. Therefore, if $h_{\ell}(n) \leq \left( \log_{\ell-1}(n)\right)^{O(1)}$ (as conjectured), then the above $\ell$-uniform analogue of the Erd\H{o}s-Hajnal conjecture fails. 
See also \cite{AST,Rodl_Schacht} for additional Erd\H{o}s-Hajnal-type results for hypergraphs.

Recently, there has been some interest in variants of the Erd\H{o}s-Hajnal problem where one forbids {\em order-size pairs} instead of induced subgraphs $H$. 
To the best of our knowledge, this setting was first considered in \cite{MS_ordersize}.
Let us define the problem. For a set $Q$ of pairs of integers, we say that an $\ell$-graph $G$ is {\em $Q$-free} if for every $(m,f) \in Q$, there is no set of $m$ vertices in $G$ which induces exactly $f$ edges. What can we say about the size of homogeneous sets in $Q$-free $\ell$-graphs? It was recently shown by Arnold, the first author, and Sudakov \cite{AGS} that this problem does in fact satisfy the natural hypergraph analogue of the Erd\H{o}s-Hajnal conjecture, in the sense that for every $\ell \geq 2$, all but a finite number\footnote{It is conjectured in \cite{AGS} that these exceptions are not necessary and the result in fact holds for every non-empty $Q$.} of the sets $Q \neq \emptyset$ satisfy that every $Q$-free $\ell$-graph on $n$ vertices contains a homogeneous set of size at least
$\left( \log_{\ell-2}(n) \right)^{c_Q}$. 

The papers \cite{ABGMW,GT} studied the special case where $\ell=3$ and $Q$ consists of pairs of the form $(4,f)$; i.e., we forbid certain numbers of edges on vertex-sets of size four. The first author and Tomon \cite{GT} showed that if an $n$-vertex 3-graph $G$ has no four vertices spanning exactly two edges, then $G$ has a homogeneous set of size $n^c$ (where $c > 0$ is an absolute constant). They also asked to show that this extends to $\ell$-graphs, in the following sense: If an $n$-vertex $\ell$-graph $G$, with $\ell \geq 3$, has no $\ell+1$ vertices spanning $2,3,4,\dots,\ell-1$ edges, then $G$ has a homogeneous set of size $n^{c_{\ell}}$. Here we prove this conjecture. We call $\ell$-graphs with this property {\em disperse}. 
 
\begin{definition}
    Let $G$ be an $\ell$-graph, $\ell \geq 3$. We say that $G$ is disperse if, for every $(\ell + 1)$-set $X = \{v_1, \dots, v_{\ell + 1}\} \subseteq V(G)$, it holds that $e_G(X) \in \{0, 1, \ell, \ell + 1\}$. 
\end{definition}
\noindent
Note that if $G$ is disperse then so is its complement $\overline{G}$.

\begin{theorem}\label{thm:EH}
    Let $G$ be a disperse $\ell$-graph on $n$ vertices. Then, $\max(\alpha(G), \omega(G)) \geq \Omega(n^{\frac{1}{3\ell - 1}})$.
\end{theorem}

The constant $\frac{1}{3\ell-1}$ in Theorem \ref{thm:EH} is best possible up to a factor of roughly $\frac{1}{3}$. Indeed, note that if an $\ell$-graph $G$ has no two edges intersecting in $\ell-1$ vertices, then $G$ is disperse (because every $(\ell+1)$-set of vertices spans 0 or 1 edges). Such $\ell$-graphs are known as {\em partial $(\ell,\ell-1)$-Steiner systems}. It is known \cite{Rodl_Sinajova} that there exist $n$-vertex partial $(\ell,\ell-1)$-Steiner systems with independence number 
$\tilde{O}(n^{\frac{1}{\ell-1}})$. Hence, the constant in Theorem \ref{thm:EH} cannot be improved beyond $\frac{1}{\ell-1}$. It is also worth noting that in the case $\ell=3$, the constant $\frac{1}{8}$ given by Theorem \ref{thm:EH} significantly improves the constant obtained in \cite{GT}, which was much smaller and therefore not calculated explicitly. It would be interesting to determine the best possible constant in Theorem \ref{thm:EH}. For $\ell=3$, the best upper bound on this constant comes from \cite{ABGMW}, which constructed a disperse $n$-vertex $3$-graph $G$ with $\max(\alpha(G), \omega(G)) \leq \tilde{O}(n^{1/3})$.\footnote{In fact, this construction avoids not only four vertices spanning two edges, but also four vertices spanning three edges.}

As mentioned above, one example of disperse $\ell$-graphs is partial $(\ell,\ell-1)$-Steiner systems. Let us now describe another important example. A hypergraph $G$ is called a {\em cohypergraph} if either $|V(G)| = 1$, or there is a vertex partition $V(G) = X \cup Y$ such that $G[X],G[Y]$ are cohypergraphs and $E(G)$ contains all or none of the $\ell$-sets which intersect both $X$ and $Y$. It is easy to check (by induction) that every cohypergraph is disperse. Also, note that for $\ell=2$, this recovers the definition of {\em cographs}. Recall also that a graph is a cograph if and only if it is induced $P_4$-free, where $P_4$ is the path with 4 vertices. 

To prove Theorem \ref{thm:EH}, we show (implicitly) that every disperse $\ell$-graph is close to being a cohypergraph. To this end, we prove several structural properties of disperse hypergraphs. Let us now state the two key facts we will use, Theorems \ref{thm:not tightly connected} and \ref{thm:tight components are cliques-1}. 
We first introduce the basic definitions related to tight connectivity. 
\begin{definition}
    Let $G$ be an $\ell$-graph.
    A {\em tight walk} is a sequence of edges
    $e_1, \dots, e_m \in E(G)$ such that $|e_i \cap e_{i + 1}| \geq \ell - 1$ for all $1 \leq i \leq m - 1$. 
    For $A,B \in \binom{V(G)}{\ell-1}$, we say that $A,B$ are connected in $G$ if there is a tight walk $e_1,\dots,e_m$ with $A \subseteq e_1$ and $B \subseteq e_m$. This is an equivalence relation.\footnote{Indeed, let $e_1,\dots,e_m$ be a tight walk with $A \subseteq e_1, B \subseteq e_m$, and let $f_1,\dots,f_k$ be a tight walk with $B \subseteq e_1, C \subseteq f_k$. Then $e_1,\dots,e_m,f_1,\dots,f_k$ is a tight walk connecting $A$ and $C$.}. A {\em tight component} of $G$ is an equivalence class in this relation. Thus, the tight components partition $\binom{V(G)}{\ell-1}$. We say that $G$ is {\em tightly connected} if there is a single tight component, namely $\binom{V(G)}{\ell-1}$.
\end{definition}
In other words, $G$ is tightly connected if for every $A,B \in \binom{V(G)}{\ell-1}$, there is a tight walk $e_1,\dots,e_m$ with $A \subseteq e_1$ and $B \subseteq e_m$. Note that in the case $\ell=2$, this coincides with the usual notion of graph connectivity. 

For a hypergraph $G$, we use $\overline{G}$ to denote the complement of $G$. Our first theorem states that for a disperse $G$, either $G$ or $\overline{G}$ is not tightly connected. Our second theorem states that in a disperse $G$, each tight component is a complete hypergraph, i.e., it consists of all $(\ell-1)$-tuples contained in some vertex set.


\begin{theorem}\label{thm:not tightly connected}
    Let $G$ be a disperse $\ell$-graph. Then $G$ or $\overline{G}$ is not tightly connected.
\end{theorem}

\begin{theorem}\label{thm:tight components are cliques-1}
    Let $G$ be a disperse $\ell$-graph. Then for every tight component $\mC$ of $G$, there is $U \subseteq V(G)$ such that $\mC = \binom{U}{\ell-1}$.
\end{theorem}

Combining these two theorems, we conclude that for every disperse $\ell$-graph $G$, either in $G$ or in $\overline{G}$ we can find a vertex set $\emptyset \neq X \subsetneq V(G)$ such that there is no edge having $\ell-1$ vertices in $X$ and one vertex in $Y := V(G) \setminus X$. Indeed, we simply take $X$ to be the vertex set corresponding to a tight component (via Theorem \ref{thm:tight components are cliques-1}).
One can then show, using the fact that $G$ is disperse, that there are only few edges intersecting both $X$ and $Y$. 
Repeating this inside $X$ and $Y$ (with some technicalities omitted here), we can show that $G$ is close to a cohypergraph. 

Theorems \ref{thm:not tightly connected} and \ref{thm:tight components are cliques-1} are proved in Section \ref{sec:structural}. We then prove Theorem \ref{thm:EH} in Section \ref{sec:EH}. Section \ref{sec:concluding} contains some further remarks and open problems. We mostly use standard graph-theoretic notation. For a vertex-set $X$ in a hypergraph $G$, we use $e_G(X)$ to denote the number of edges of $G$ contained in $X$. Throughout the rest of the paper, we assume that the uniformity $\ell$ is at least 3, unless explicitly stated otherwise.

\section{Structural Results}\label{sec:structural}



In this section we study the structure of disperse hypergraphs, and in particular prove Theorems \ref{thm:not tightly connected} and \ref{thm:tight components are cliques-1}.
Let us begin with some lemmas. 
For an $\ell$-graph $G$ and $v \in V(G)$, the {\em link} $L(v)$ is the $(\ell-1)$-graph on $V(G) \setminus \{v\}$ with edge set 
$\{ e \in \nolinebreak \binom{V(G) \setminus \{v\}}{\ell-1} : e \cup \{v\} \in E(G) \}$. A useful fact is that the links of a disperse hypergraph are also disperse.

\begin{lemma}\label{lem:links are disperse}
    Let $G$ be a disperse $\ell$-graph. Then for every $v \in V(G)$, $L(v)$ is disperse.
\end{lemma}

\begin{proof}
    Suppose that there is $v \in G$ such that $L(v)$ is not disperse, and let $X = \{v_1, \dots, v_{\ell}\} \subseteq V \setminus \{v\}$ witness this, namely, 
    $2 \leq |e_{L(v)}(X)| \leq \ell - 2$. 
    Then $2 \leq e_G(X \cup \{v\}) \leq \ell-1$, contradicting that $G$ is disperse. 
\end{proof}
\noindent
We will also use the following easy lemma.
\begin{lemma}\label{lem:2-edge tight walk}
    Let $G$ be a disperse $\ell$-graph, let $f,g \in E(G)$ with $|f \cap g| \geq \ell-1$, let $v \in f$ and let $A,B \subseteq g \setminus \{v\}$ with $|A| = |B| = \ell-2$. Then $A,B$ are connected in $L(v)$. 
\end{lemma}
\begin{proof}
    If $v \in g$ then the assertion clearly holds, because $g \setminus \{v\}$ is a (one-edge) tight walk between $A$ and $B$ in $L(v)$. So suppose that $v \notin g$. Then we can write $f = \{v,v_1,\dots,v_{\ell-1}\}, g = \{v_1,\dots,v_{\ell-1},w\}$ with $w \neq v$. As $G$ is disperse, the $(\ell+1)$-set $f \cup g$ must miss at most one edge. This means that there are $e_1,e_2 \in E(G)$ with $\{v\} \cup A \subseteq e_1$ and $\{v\} \cup B \subseteq e_2$. Now $e_1 \setminus \{v\},e_2 \setminus \{v\}$ is a tight walk between $A$ and $B$ in $L(v)$.  
\end{proof}

The following is our main technical lemma. It shows that in a disperse hypergraph, if there is a tight walk of length 3 from a vertex $v$ to an $(\ell-1)$-set $B$, then there is also such a tight walk of length 2. 

\begin{lemma} \label{tight-link-lemma-1}
    Let $G$ be a disperse $\ell$-graph, let $e,f,g$ be a tight walk in $G$, let $v \in e$, and let $B \subseteq g$ with $|B| = \ell-1$. Then there is a tight walk $f',g'$ in $G$ such that $v \in f'$ and $B \subseteq g'$.
\end{lemma}

\begin{proof}
    We first consider some easy degenerate cases. 
    If $v \in f$ or $v \in g$ then we can take $(f',g') = (f,g)$ or $(f',g') = (g,g)$, respectively. 
    If $B \subseteq f$ then take $(f',g') = (e,f)$, and if 
    $|e \cap g| \geq \ell-1$ then take $(f',g') = (e,g)$. Assuming that none of the above holds, without loss of generality we may write $e = \{v, w_1, \dots, w_{\ell - 1}\}, f = \{w_1, \dots, w_{\ell}\}$ and $g = \{w_2, \dots, w_{\ell + 1}\}$ with the vertices $v,w_1,\dots,w_{\ell+1}$ being distinct, and $B = \{w_2, \dots, w_{\ell + 1}\} \backslash \{w_i\}$ for some $i \neq \ell + 1$. Note that $|f \cap B| = \ell-2$.
    
    Suppose, for the sake of contradiction, that there is no tight walk $f',g'$ as in the lemma. 
    We proceed via a series of claims which will eventually give a contradiction.

    \begin{claim}\label{claim-1-generalized-1}
        For all $A \subseteq f \backslash \{w_1\}$ with $|A| = \ell - 2$, it holds that $\{v, w_1\} \cup A \in E(G)$.
    \end{claim}

    \begin{proof}
        Consider the $(\ell+1)$-set $X = \{v, w_1, ..., w_{\ell}\}$. As $e,f$ are both edges of $G$ contained in $X$, and as $G$ is disperse, $G$ misses at most one edge on $X$. Note that $h := \{v,w_2,\dots,w_{\ell}\}$ is not an edge of $G$, as otherwise $f' = h$ and $g' = g$ satisfy $v \in f'$, $|f' \cap g'| = |\{w_2, \dots, w_{\ell}\}| = \ell - 1$ and $B \subseteq g'$. So, $G$ contains all edges on $X$ besides $h$; in particular, $G$ contains $\{v, w_1\} \cup A$ for all $A \subseteq f \backslash \{w_1\}$ of size $\ell - 2$.
    \end{proof}

    \begin{claim}\label{claim-2-generalized-1}
        For all $A \subseteq g$ with $|A| = \ell - 1$, it holds that $\{v\} \cup A \not \in E(G)$.
    \end{claim}

    \begin{proof}
        Otherwise we can take $f' = \{v\} \cup A$ and $g' = g$,  which satisfy $v \in f'$, $|f' \cap g'| = |A| = \ell - 1$ and $B \subseteq g'$. 
    \end{proof}

    \begin{claim}\label{claim-3-generalized-1}
        For all $A \subseteq g$ with $|A| = \ell - 1$, it holds that $\{w_1\} \cup A \in E(G)$ if and only if $A \neq B$.
    \end{claim}

    \begin{proof}
        Let $X = \{w_1, ..., w_{\ell + 1}\}$. Note that $f, g$ are edges of $G$ contained in $X$, and thus, $G$ misses at most one edge on $X$. Moreover, if $\{w_1\} \cup B \in E(G)$, then letting $f' = \{v, w_1\} \cup (f \cap B)$ (which is an edge by Claim \ref{claim-1-generalized-1}) and $g' = \{w_1\} \cup B$ gives the desired result. Thus, $\{w_1\} \cup B \notin E(G)$, implying that $G$ must contain all other edges on $X$ -- in particular, all edges of the form $\{w_1\} \cup A$ with $A \subseteq g$, $|A| = \ell - 1$ and $A \neq B$.  
    \end{proof}

    \begin{claim}\label{claim-4-generalized-1}
        For all $A \subseteq B$ with $|A| = \ell - 2$ and $A \neq f \cap B$, it holds that $\{v, w_1\} \cup A \not \in E(G)$.
    \end{claim}

    \begin{proof}
        Let $X = \{v, w_1\} \cup B$, so $|X| = \ell+1$. 
        By Claim \ref{claim-2-generalized-1} we have 
        $\{v\} \cup B \not \in E(G)$, and by Claim
        \ref{claim-3-generalized-1} we have $\{w_1\} \cup B \notin E(G)$. 
        Hence, as $G$ is disperse, it can contain at most one edge on $X$, which must be the edge 
        $\{v, w_1\} \cup (f \cap B)$, as this is an edge by Claim \ref{claim-1-generalized-1}. So, $G$ cannot contain any edge of the form $\{v, w_1\} \cup A$ for $A \subseteq B$ with $|A| = \ell - 2$ and $A \neq f \cap B$. 
    \end{proof}


    Now, we use the above claims to derive a contradiction and hence prove Lemma \ref{tight-link-lemma-1}.
    Fix an arbitrary $x \in f \cap B$ (this is possible because $|f \cap B| = \ell-2 \geq 1$). 
    Set
    $C = g \backslash \{x\}$, and note that $|C| = \ell-1$; $C \neq B$ and hence $|C \cap B| = \ell-2$; $w_{\ell+1} \in C$ (as $w_{\ell+1} \notin f$ and so $x \neq w_{\ell+1}$);
    and $C \cap B \neq f \cap B$ (as $x \in f \cap B$ while $x \notin C$). 
    Consider the set $X = \{v, w_1\} \cup C$, so $|X| = \ell+1$. We will show that there exist two edges and two non-edges on $X$. Indeed, by Claim \ref{claim-1-generalized-1}, 
    $\{v, w_1\} \cup (C \backslash \{w_{\ell + 1}\}) \in E(G)$, and by Claim \ref{claim-3-generalized-1}, $\{w_1\} \cup C \in E(G)$. On the other hand, by Claim \ref{claim-2-generalized-1}, $\{v\} \cup C \not \in E(G)$, and by Claim \ref{claim-4-generalized-1}, $\{v, w_1\} \cup (C \cap B) \not \in E(G)$. This contradicts the assumption that $G$ is disperse, as desired.
\end{proof}

\noindent
Lemma \ref{tight-link-lemma-1} easily implies the following.

\begin{lemma}\label{corollary-walks}
    Let $G$ be a disperse $\ell$-graph, let $e_1, \dots, e_m$ be a tight walk in $G$, and let $v \in e_1$. 
    Then for all $1 \leq i \leq m - 1$, there exists a tight walk $f, g$ such that $v \in f$ and $e_i \cap e_{i + 1} \subseteq g$.
\end{lemma}

\begin{proof}
    We prove this by induction on $i$. For $i = 1$ the claim is trivial by letting $f = g = e_1$. For the inductive step, let $f, g$ be a tight walk with $v \in f$ and $e_{i - 1} \cap e_i \subseteq g$. Note that $f, g, e_i$ is a tight walk as $f,g$ is a tight walk and $|g \cap e_i| \geq |e_{i - 1} \cap e_i| = \ell - 1$. Moreover, $v \in f$ and $B := e_i \cap e_{i + 1} \subseteq e_i$, so we may apply Lemma \ref{tight-link-lemma-1} to obtain a tight walk $f',g'$ with $v \in f'$ and $e_i \cap e_{i + 1} \subseteq g'$, as desired. 
\end{proof}

\noindent
We now use this result to prove the following key lemma.

\begin{lemma}\label{main-lemma-1}
    Let $G$ be a disperse $\ell$-graph, let $v \in V(G)$ and let $A,B \subseteq V(G) \setminus \{v\}$ with $|A| = |B| = \ell-2$. 
    If $A \cup \{v\}$ and $B \cup \{v\}$ are connected in $G$, then $A$ and $B$ are connected in $L(v)$.
\end{lemma}

\begin{proof}
    
    Let $e_1,\dots,e_m$ be a tight walk in $G$ between $A \cup \{v\}$ and $B \cup \{v\}$. 
    We will prove the lemma by induction on $m$. 
    For $m \leq 2$ the lemma is trivial, by simply observing that $e_1 \backslash \{v\}, e_m \backslash \{v\}$ is a tight walk between $A$ and $B$ in $L(v)$. So, suppose $m \geq 3$. First, consider the case where $v \in e_i$ for some $1 < i < m$. 
    Fix any $C \subseteq e_i \setminus \{v\}$ of size $\ell - 2$.
    By the inductive hypothesis, $A$ and $C$ are connected in $L(v)$, because $e_1,\dots,e_i$ is a tight walk between $A \cup \{v\}$ and $C \cup \{v\}$ in $G$. Similarly, $C$ and $B$ are connected in $L(v)$, 
    by applying the inductive hypothesis to $e_i,\dots,e_m$.
    By transitivity, $A$ and $B$ are connected in $L(v)$. 
    
    Now, suppose that $v \not \in e_i$ for every $1 < i < m$. 
    For $1 \leq i \leq m-2$, let $C_i := e_i \cap e_{i+1} \cap e_{i+2}$. Then $|C_i| \geq \ell-2$ and we may assume, by passing to a subset if necessary, that $|C_i| = \ell-2$. Note that $v \notin C_i$ for every $1 \leq i \leq m-2$. 
    Set also $C_0 = A$ and $C_{m-1} = B$. 
    It suffices to show that for every $1 \leq i \leq m-1$, $C_{i-1},C_{i}$ are connected in $L(v)$. Indeed, this would imply, by transitivity, that $A = C_0$ and $B = C_{m-1}$ are connected in $L(v)$.
    Observe that $A$ and $C_1$ are connected in $L(v)$ because $e_1 \setminus \{v\}$ is an edge of $L(v)$ containing $A, C_1$.
    Similarly, $B$ and $C_{m-2}$ are connected in $L(v)$ because  $v \in e_m$ and $B,C_{m-2} \subseteq e_m \setminus \{v\}$. 
    Now let $2 \leq i \leq m-2$. 
    By Lemma \ref{corollary-walks}, there is a tight walk $f, g$ with $v \in f$ and $e_i \cap e_{i + 1} \subseteq g$. 
    By definition, $C_i,C_{i-1} \subseteq e_i \cap e_{i+1}$, and hence $C_i,C_{i-1} \subseteq g$. Now, by Lemma \ref{lem:2-edge tight walk}, $C_i,C_{i-1}$ are connected in $L(v)$, as required.
\end{proof}

\noindent
Lemma \ref{main-lemma-1} immediately implies the following.

\begin{theorem}\label{links-are-tightly-connected-1}
    Let $G$ be a tightly connected disperse $\ell$-graph. Then $L(v)$ is tightly connected for all $v \in V(G)$. 
\end{theorem}

\begin{proof}
    Let $v \in V(G)$. Fix any $A,B \subseteq V(G) \setminus \{v\}$ of size $\ell-2$. As $G$ is tightly connected, $A \cup \{v\}$ and $B \cup \{v\}$ are connected in $G$. Hence, by Lemma \ref{main-lemma-1}, $A$ and $B$ are connected in $L(v)$. This shows that $L(v)$ is tightly connected.
\end{proof}

Next, we need the following simple lemma proved by the first author and Tomon \cite{GT}. For completeness, we include the proof. 

\begin{lemma}\label{gishboliner-and-tomon-result}
    Let $G$ be a disperse $3$-graph. Then $L(v)$ is a cograph for every $v \in V(G)$.
\end{lemma}
\begin{proof}
Let $v \in V(G)$, and suppose by contradiction that $a,b,c,d$ is an induced path in $L(v)$. In order for $v,a,b,c$ not to span two edges, we must have $\{a,b,c\} \in E(G)$. Similarly, $\{b,c,d\} \in E(G)$ and $\{a,b,d\},\{a,c,d\} \notin E(G)$. But now $a,b,c,d$ span two edges, a contradiction.  
\end{proof}

\noindent
We now proceed to the proof of Theorem \ref{thm:not tightly connected}.

\begin{proof}[Proof of Theorem \ref{thm:not tightly connected}]
    We use induction on $\ell$. When $\ell = 3$, the link of each vertex is a cograph by Lemma \ref{gishboliner-and-tomon-result}, and every cograph $H$ satisfies that $H$ or $\overline{H}$ is not connected. In particular, by the contrapositive of Theorem \ref{links-are-tightly-connected-1}, this implies that $G$ or $\overline{G}$ is not tightly connected. Now, suppose $\ell > 3$. 
    Fix any $v \in V(G)$. By Lemma \ref{lem:links are disperse}, $L_G(v)$ is disperse. Hence, by the induction hypothesis, $L_G(v)$ or $\overline{L_G(v)} = L_{\overline{G}}(v)$ is not tightly connected. Thus, $G$ or $\overline{G}$ is not tightly connected by the contrapositive of Theorem \ref{links-are-tightly-connected-1}, as desired.
\end{proof}

Next, we consider the structure of tight components in disperse hypergraphs, with the goal of proving Theorem \ref{thm:tight components are cliques-1}. We begin a sequence of lemmas.  



\begin{lemma}\label{length-2-lemma-1}
    Let $G$ be a disperse $\ell$-graph and let $A_1,A_2 \subseteq V(G)$ with $|A_1| = |A_2| = \ell-1$ and $|A_1 \cap A_2| = \ell-2$.
    Suppose that $G$ has a tight walk between $A_1$ and $A_2$. Then there exists such a tight walk on two or less edges.
\end{lemma}

\begin{proof}
    We prove this claim by induction on the uniformity $\ell$. First, suppose $\ell = 3$ and write $A_1 = \{v,x\}$, $A_2 = \{v,y\}$. By Lemma \ref{main-lemma-1}, $x$ and $y$ are in the same component of $L(v)$. Since $L(v)$ is a cograph by Lemma \ref{gishboliner-and-tomon-result}, $L(v)$ is induced $P_4$-free, and so, the shortest path from $x$ to $y$ in $L(v)$ has length at most 2. Thus, $G$ contains a tight walk from $\{v,x\}$ to $\{v,y\}$ of length at most 2, as required. Now, suppose $\ell \geq 4$. Let $v \in A_1 \cap A_2$ be arbitrary. By Lemma \ref{main-lemma-1}, there is a tight walk from $A_1 \setminus \{v\}$ to $A_2 \setminus \{v\}$ in $L(v)$. Hence, by the inductive hypothesis, there exists such a walk of length at most 2. Denote this walk by $e_1, e_2$ (possibly $e_1=e_2$). Then, $e_1 \cup \{v\}, e_2 \cup \{v\}$ is a tight walk of length at most 2 from $A_1$ to $A_2$ in $G$, as desired.
\end{proof}

\begin{lemma}\label{dec-27-lemma-1}
    Let $G$ be a disperse $\ell$-graph, 
    let $B \subseteq V(G)$ of size $\ell$, and let $A_1,A_2 \subseteq B$ be distinct subsets of size $\ell-1$. 
    If $A_1,A_2$ belong to the same tight component $\mC$ of $G$, then every $(\ell - 1)$-subset of $B$ also belongs to $\mC$. 
\end{lemma}

\begin{proof}
    By assumption, $G$ has a tight walk between $A_1$ and $A_2$. 
    By Lemma \ref{length-2-lemma-1}, there exists such a tight walk $W$ of length at most 2. If $W$ has length 1 then it consists of the single edge $e = A_1 \cup A_2 = B$, and thus, any $(\ell - 1)$-subset of $B$ clearly belongs to $\mC$. So suppose that $W$ has length 2 and write $W = (e_1,e_2)$, where 
    $A_1 \subseteq e_1$, $A_2 \subseteq e_2$ and $e_1 \neq e_2$. 
    Since $G$ is disperse, $G$ misses at most one edge on the $(\ell+1)$-set $C := e_1 \cup e_2$. 
    This implies that for every $(\ell-1)$-set $A \subseteq C$, there is an edge $e \in E(G)$ with $A \subseteq e \subseteq C$ (indeed, otherwise $C$ misses at least two edges). As $|e \cap e_1| \geq \ell-1$, we have a tight walk $(e_1,e)$ between $A_1$ and $A$. Hence, $A$ belongs to $\mC$, as required.  
\end{proof}


\begin{lemma}\label{dec-27-lemma-2-1}
    Let $G$ be a disperse $\ell$-graph and let $e_1, ..., e_m$ be a tight walk in $G$. Then all $(\ell - 1)$-subsets of $\bigcup_{i = 1}^me_i$ belong to the same tight component of $G$.
\end{lemma}

\begin{proof}
    We will prove this claim by induction on $m$. The claim is trivial when $m = 1$, as $e_1$ is a tight walk between any two $(\ell - 1)$-subsets of $e_1$. Now assume $m \geq 2$.
    By the induction hypothesis, there is a tight component $\mC$ containing all $(\ell - 1)$-subsets of $W := \bigcup_{i = 1}^{m - 1}e_i$. 
    If $e_m \subseteq \bigcup_{i=1}^{m-1} e_i$ then there is nothing to prove, so suppose otherwise. Let $x$ be the unique vertex in $e_m \setminus e_{m-1}$, and write 
    $e_m = \{v_1, v_2, \dots, v_{\ell - 1}, x\}$.
    It suffices to show that 
    for every $w_1,\dots,w_{\ell-2} \in W$, the
    $(\ell - 1)$-set 
    $\{w_1, w_2, ..., w_{\ell - 2}, x\}$ belongs to $\mC$.
    We now prove by induction on $j$ that for every $0 \leq j \leq \ell-2$,
    $A_j := \{w_1, ..., w_j, v_{j + 1}, ..., v_{\ell - 2}, x\} \in \mC$.
    For the base case $j=0$, note that $A_0 = \{v_1,\dots,v_{\ell-2},x\}$ is in the same tight component as $\{v_1,\dots,v_{\ell-1}\} \in \mC$, because both of these $(\ell-1)$-sets are contained in $e_m$. For the induction step, let $1 \leq j \leq \ell-2$. By the inductive hypothesis, 
    $A_{j-1} = \{w_1, ..., w_{j - 1}, v_j, v_{j + 1}, ..., v_{\ell - 2}, x\} \in \mC$.
    Also, $\{w_1, ..., w_j, v_j, v_{j + 1}, ..., v_{\ell - 2}\} \in \mC$ because all vertices of this $(\ell-1)$-set belong to $W$. 
    Hence, by Lemma \ref{dec-27-lemma-1} for the set 
    $B := \{w_1, ..., w_j, v_j, v_{j + 1}, ..., v_{\ell - 2}, x\}$,
    every $(\ell-1)$-subset of $B$, and in particular $A_j$, also belongs to $\mC$. This completes the induction step. Taking $j = \ell-2$, we get that $A_{\ell-2} = \{w_1,\dots,w_{\ell-2},x\} \in \mC$, as required.  
\end{proof}

\noindent
Using the above lemmas, we can now prove Theorem \ref{thm:tight components are cliques-1}.

\begin{proof}[Proof of Theorem \ref{thm:tight components are cliques-1}]
    
    Let $\mC$ be a tight component of $G$. 
    Let $U$ be the set of all $v \in V(G)$ such that $v \in A$ for some $A \in \mC$. We want to show that $\mC = \binom{U}{\ell-1}$. 
    So suppose for the sake of contradiction that there exists some $A \subseteq U$ of size $\ell-1$ such that $A \not \in \mC$. Let $A'$ be a largest subset of $A$ such that $A' \subseteq B$ for some $B \in \mC$. 
    As $A \notin \mC$, we have $|A'| < |A|$. So fix $x \in A \setminus A'$. 
    As $x \in U$, there is some $C \in \mC$ with $x \in C$ (by the definition of $U$).
    As $B,C \in \mC$, there is a tight walk $e_1,\dots,e_m$ in $G$ between $B$ and $C$. 
    By Lemma \ref{dec-27-lemma-2-1}, each $(\ell-1)$-tuple of vertices in $\bigcup_{i=1}^m e_i$ also belongs to $\mC$. 
    As $A' \cup \{x\} \subseteq B \cup C \subseteq \bigcup_{i=1}^m e_i$, there is an $(\ell-1)$-tuple $D \in \mC$ with $A' \cup \{x\} \subseteq D$. However, this contradicts the maximality of $A'$, completing the proof. 
\end{proof}




\section{Proof of Theorem \ref{thm:EH}}\label{sec:EH}

The following is a standard bound on the independence number of hypergraphs, due to Spencer \cite{Spencer}.

\begin{lemma}\label{spencer-ind-set-bound}
    Let $G$ be an $\ell$-graph with $n$ vertices and average degree $d$. Then 
    $$
    \alpha(G) \geq \frac{\ell-1}{\ell}n \cdot \min \{ 1,d^{-1/(\ell-1)} \}.
    $$
\end{lemma}
\begin{proof}
    If $d < 1$ then $e(G) = \frac{nd}{\ell} < \frac{n}{\ell}$, and deleting one vertex per edge gives an independent set of size at least $\frac{\ell-1}{\ell}n$.
    Suppose now that $d \geq 1$. Sample a random subset $U \subseteq V(G)$ with probability $p := d^{-1/(\ell-1)} \leq 1$. Deleting one vertex per edge in $U$ gives an independent set of size at least $|U| - e(U)$. By linearity of expectation, we have $\mathbb{E}[|U| - e(U)] = pn - p^{\ell} \frac{dn}{\ell} = pn(1 - \frac{p^{\ell-1}d}{\ell}) = \frac{\ell-1}{\ell}pn = \frac{\ell-1}{\ell}nd^{-1/(\ell-1)}$. 
\end{proof}

It is well-known that every $n$-vertex cograph contains a homogeneous set of size at least $n^{1/2}$. The same proof applies to cohypergraphs, as follows. 

\begin{lemma}\label{homogeneous-bound-1}
        Every cohypergraph $G$ on $n$ vertices satisfies $\alpha(G) \cdot \omega(G) \geq n$ and hence 
        $\max(\alpha(G), \omega(G)) \geq n^{\frac{1}{2}}$.
    \end{lemma}

    \begin{proof}
        We prove this by induction on $n$. 
        The case $n=1$ is trivial, so suppose $n \geq 2$. By definition, we can write $V(G) = V(H_1) \cup V(H_2)$, where $H_1$ and $H_2$ are vertex-disjoint cohypergraphs, and $G$ has either all edges or no edges which intersect both $V(H_1)$ and $V(H_2)$. By the inductive hypothesis, 
        $\alpha(H_i) \cdot \nolinebreak \omega(H_i) \geq |V(H_i)|$ for $i=1,2$.
        Let us assume that $G$ has all edges which intersect both $V(H_1)$ and $V(H_2)$; the other case is symmetrical (by switching to $\overline{G}$). Then $\omega(G) = \omega(H_1) + \omega(H_2)$ and $\alpha(G) = \max(\alpha(H_1), \alpha(H_2))$, giving 
        \begin{align*}
            \alpha(G) \cdot \omega(G) &= \max(\alpha(H_1), \alpha(H_2)) \cdot (\omega(H_1) + \omega(H_2)) \\
            &\geq \alpha(H_1) \cdot \omega(H_1) + \alpha(H_2) \cdot \omega(H_2) \\
            &\geq |V(H_1)| + |V(H_2)| \\
            &= |V(G)| = n,
        \end{align*}
        as desired.
    \end{proof}

For a tight component $\mC$ of a disperse $\ell$-graph $G$, we denote by $V(\mC)$ the vertex set $U$ satisfying $\mC = \binom{U}{\ell-1}$, using Theorem \ref{thm:tight components are cliques-1}. The following lemma shows that if all tight components of $G$ are small, then $G$ has a large independent set.  

\begin{lemma}\label{case-one-bound}
    Let $G$ be a disperse $\ell$-graph on $n$ vertices, and suppose that for every tight component $\mC$ of $G$ it holds that $|V(\mC)| \leq m$. Then $\alpha(G) \geq \frac{\ell-1}{\ell} \cdot (n/m)^{\frac{1}{\ell-1}}$.
\end{lemma}

\begin{proof}
    First, we will bound the number of edges of $G$. Note that each edge of $G$ is contained in $V(\mC)$ for some tight component $\mC$. Thus, by summing over all tight components, we obtain that
    \begin{align*}
        e(G) &\leq \sum_{\mC}\binom{|V(\mC)|}{\ell} \\
        &\leq \frac{m}{\ell}\sum_{\mC}\binom{|V(\mC)|}{\ell - 1} \\
        &= \frac{m}{\ell}\binom{n}{\ell - 1},
    \end{align*}
    where the equality uses the fact that every $(\ell-1)$-subset of $V(G)$ is contained in $V(\mC)$ for exactly one tight component $\mC$, as the tight components partition $\binom{V(G)}{\ell-1}$. 
    Letting $d$ denote the average degree of $G$, we get 
    $$
    d = \frac{\ell \cdot  e(G)}{n} \leq \frac{m \binom{n}{\ell - 1}}{n} \leq m n^{\ell-2}.
    $$
    Thus, by Lemma \ref{spencer-ind-set-bound}, 
    \begin{align*}
        \alpha(G) &\geq \frac{\ell - 1}{\ell} \cdot \frac{n}{m^{\frac{1}{\ell-1}}n^{\frac{\ell-2}{\ell-1}}} = 
        \frac{\ell - 1}{\ell} \cdot \left( \frac{n}{m} \right)^{\frac{1}{\ell-1}},
    \end{align*}
    as desired. 
\end{proof}

     For two disjoint vertex sets $X,Y$ in an $\ell$-graph $G$ and for $1 \leq i \leq \ell-1$, denote by $E_G(X^i,Y^{\ell-i})$ the set of edges of $G$ having $i$ vertices in $X$ and $\ell-i$ vertices in $Y$, and let $e_G(X^i,Y^{\ell-i})$ be the number of such edges. An edge in $\bigcup_{i=1}^{\ell-1} E_G(X^i,Y^{\ell-i})$ is said to {\em cross} $(X,Y)$.

    \begin{lemma}\label{bound-for-c-1}
        Let $G$ be a disperse $\ell$-graph, and
        let $X,Y \subseteq V(G)$ be disjoint vertex sets with 
        $e_G(X^{\ell-1},Y^1) = 0$. 
        Let $A = \{v_1, \dots, v_{\ell - 1}\} \subseteq X \cup Y$, denote $i = |A \cap Y|$ and assume $i \geq 1$. Then there are 
        less than $2^{i - 1}$ edges in $E_G(X^{\ell-i},Y^i)$ containing $A$. 
    \end{lemma}

    \begin{proof}
        We prove the claim by induction on $i$. The base case $i=1$ holds by assumption as $e_G(X^{\ell - 1}, Y^1) = 0$. Now let $2 \leq i \leq \ell-1$ and suppose that the result holds for $i-1$. 
        Without loss of generality, assume that $v_1,v_{\ell-1} \in Y$ (note that $|A \cap Y| = i \geq 2$).
        Suppose, for the sake of contradiction, that there are at least $2^{i - 1}$ edges in $E_G(X^{\ell-i},Y^i)$ which contain $A$. As $|A \cap Y| = i$, each such edge consists of $A$ and a vertex from $X$. So let $X_A = \{x \in X \setminus A : \{x\} \cup A \in E(G)\}$ denote the set of vertices in $X$ such that $\{x\} \cup A \in E_G(X^{\ell-i},Y^i)$. Then $|X_A| \geq 2^{i - 1}$. Observe that for any $x_1, x_2 \in X_A$, the set of vertices $\{x_1, x_2\} \cup A$ contains at least two edges (namely, the edges $\{x_j\} \cup A$ for $j=1,2$), and thus, avoids at most one edge (as $G$ is disperse). 
        Color the pair $x_1x_2 \in \binom{X_A}{2}$ blue if $\{x_1, x_2, v_2, \dots, v_{\ell-1}\} \in E(G)$, and red otherwise. So, if $x_1x_2$ is colored red, then $\{x_1, x_2, v_1, \dots, v_{\ell - 2}\} \in E(G)$. Fix any $v \in X_A$. By pigeonhole, at least $\lceil \frac{|X_A|-1}{2} \rceil \geq 2^{i-2}$ of the edges touching $v$ have the same color. 
        Without loss of generality, suppose this color is red. That is, $\{v, x, v_1, \dots, v_{\ell - 2}\} \in E(G)$ for at least $2^{i-2}$ vertices $x$. 
        Note that the set $A' := \{v_1, \dots, v_{\ell - 2}, v\}$ satisfies $|A' \cap Y| = i-1$, because $v_{\ell-1} \in Y$ and $v \in X$. 
        Hence, each edge $\{v, x, v_1, \dots, v_{\ell - 2}\}$ as above belongs to $E_G(X^{\ell - i + 1}, Y^{i - 1})$.
        Therefore, the number of edges in $E_G(X^{\ell - i + 1}, Y^{i - 1})$ containing $A'$ is at least $2^{i - 2}$, 
        a contradiction to the induction hypothesis.  
    \end{proof}
    

\noindent
Finally, we prove Theorem \ref{thm:EH}. 

\begin{proof}[Proof of Theorem \ref{thm:EH}]
    Set $\epsilon = \frac{\ell + 1}{3\ell - 1}, \gamma = \frac{2}{3\ell - 1}$. We first define an algorithm to maintain a partition $\mathcal{P}$ of $V(G)$ and a set of $\ell$-tuples $\mathcal{B} \subseteq \binom{V(G)}{\ell}$, as follows.
    
    \begin{algorithm}[H]
    \caption{PARTITION-ALGORITHM}
    \begin{algorithmic}[1]
        \State \texttt{Initialize $\mathcal{P} = \{V(G)\}$}
        \While {\texttt{$\exists W \in \mathcal{P} \text{ such that } |W| \geq n^{1 - \gamma}$},}
            \State 
            \texttt{\text{Let $H \in \big\{G[W], \overline{G[W]}\big\}$ such that $H$ is not tightly connected (Theorem \ref{thm:not tightly connected})}}
            \If{\texttt{\text{All tight components} $\mC$ \text{of} $H$ \text{satisfy} $|V(\mC)| \leq n^{1 - \epsilon}$,}}
                \State \texttt{\text{Terminate}}
            \Else
                \State \texttt{\text{Partition} $W = X \cup Y$ \text{where} $|X| \geq n^{1 - \epsilon}$ \text{and} $X = V(\mC)$ \text{for a tight component $\mC$ of} $H$}
                \State \texttt{\text{Replace} $W$ \text{with} $X, Y$ in $\mathcal{P}$}
                \State \texttt{\text{Add to} $\mathcal{B}$ all edges in $E_{H}(X^i,Y^{\ell-i})$ for every $1 \leq i \leq \ell-2$}
            \EndIf
        \EndWhile
        \For{\texttt{$W \in \mathcal{P}$}}
            \State \texttt{\text{Add} $\binom{W}{\ell}$ \text{to} $\mathcal{B}$}
        \EndFor
    \end{algorithmic}
    \end{algorithm}

    \noindent
    We first give the following claim, which will be used later.

    \begin{claim}\label{extracted-claim}
        Let $U \subseteq V(G)$. If $U$ contains no $\ell$-tuple from $\mathcal{B}$, then $G[U]$ is a cohypergraph.
    \end{claim}

    \begin{proof}
        Observe that at each step of the algorithm, the hypergraph $H$ (Line 3) has no edges which cross $(X \cap U, Y \cap U)$ (because all such edges are added to $\mB$). This implies that $G$ contains all or none of the edges crossing $(X \cap U, Y \cap U)$ (by the choice of $H$). Moreover, $|U \cap W| \leq \ell-1$ for every $W \in \mathcal{P}$ (here $\mathcal{P}$ is the partition at the end of the algorithm), because $U$ contains no $\ell$-tuple from $\mB$, whereas $\binom{W}{\ell} \subseteq \mB$. It follows that $G[U]$ is a cohypergraph obtained by starting with the empty hypergraphs $G[U \cap W]$, $W \in \mathcal{P}$, and repeatedly joining two of the parts (by backtracking the algorithm).
    \end{proof}
        

    Note that in Line 9 of the algorithm, we have $E_H(X^{\ell-1},Y^1) = \emptyset$, because $X$ is the vertex set of a tight component $\mC$ of $H$. Indeed, if there were an edge $\{x_1,\dots,x_{\ell-1},y\} \in E_H(X^{\ell-1},Y^1)$, then we would have $y \in V(\mC) = X$, a contradiction.

    If the algorithm terminates in the first condition (Lines 4-5), then there is some set $W \subseteq V(G)$ with $|W| \geq n^{1 - \gamma}$ such that in either $G[W]$ or $\overline{G[W]}$, all tight components have size at most $n^{1-\epsilon}$. Thus, by Lemma \ref{case-one-bound} (with $m = n^{1-\epsilon}$), we have
    \begin{align*}
        \max(\alpha(G), \omega(G)) 
        &\geq 
        \frac{\ell-1}{\ell} \cdot \left( \frac{|W|}{m} \right)^{\frac{1}{\ell-1}} \\
        &\geq 0.5 n^{\frac{\epsilon - \gamma}{\ell - 1}} \\
        &= 0.5n^{\frac{1}{3\ell - 1}},
    \end{align*}
    using our choice of $\epsilon,\gamma$.
    Hence, we may assume from now on that the algorithm does not terminate on the first condition.  
    
    Thus, by Claim \ref{extracted-claim}, our goal from now on is to find a large set $U \subseteq V(G)$ which is independent in the hypergraph with edge-set $\mB$. To this end, we upper-bound $|\mB|$. 
    Let $\mB_1$ (resp. $\mB_2$) be the set of $\ell$-tuples added to $\mB$ in Line 9 (resp. Line 13) of the algorithm. 

    \begin{claim}\label{bound-1}
        $|\mB_1| \leq 2^{\ell} n^{\ell - 1 + \epsilon}$.
    \end{claim}

    \begin{proof}
        We will bound $|\mB_1|$ as follows. For each pair of sets $(X,Y)$ obtained in Line 7 of the algorithm, and each edge of $H$ $e$ which crosses $(X,Y)$, we assign to $e$ an $(\ell-1)$-subset $A \subseteq e$ with $A \cap Y = e \cap Y$ (if there are several such $A$, namely if $|e \cap Y| \leq \ell-2$, then choose one such $A$ arbitrarily). In other words, for any given $(\ell-1$)-set $A$, we will bound the number of edges $e$ of $H$ crossing $(X,Y)$ which contain $A$ and satisfy $A \cap Y = e \cap Y$. Summing over all $A \in \binom{V(G)}{\ell-1}$ and all choices for $(X,Y)$ will give us the desired bound on \nolinebreak $|\mB_1|$.
        
        So let $A \subseteq V(G)$ be an arbitrary set of size $\ell-1$. Let $(X_1, Y_1), \dots, (X_t, Y_t)$ be all pairs of sets $(X,Y)$ obtained in the course of the algorithm for which $A$ crosses $(X,Y)$, 
        ordered according to when the algorithm processed them. 
        Since any two sets handled by the algorithm are disjoint or nested, and since all $Y_i$ intersect $A$, we must have
        $Y_1 \supseteq Y_2 \supseteq \dots \supseteq Y_t$. 
        Also, since $|X_i| \geq n^{1-\epsilon}$ for every $i$ (see Line 7 of the algorithm), we have $|Y_i| \leq |Y_{i - 1}| - n^{1 - \epsilon}$ for all $2 \leq i \leq t$, and thus, $t \leq n^{\epsilon}$. Now, fix any $1 \leq i \leq t$. As $e_H(X_i^{\ell-1},Y_i^1) = 0$, Lemma \ref{bound-for-c-1} implies that the number of edges $e$ of $H$ crossing $(X_i,Y_i)$ and satisfying $A \cap Y_i = e \cap Y_i$ is less than $\sum_{i=1}^{\ell-1} 2^{i - 1} = 2^{\ell - 1} - 1$. Summing over all choices of $A \in \binom{V(G)}{\ell-1}$ and $i=1,\dots,t$, we get that


        \begin{align*}
            |\mB_1| &\leq \binom{n}{\ell-1} \cdot n^{\epsilon} \cdot 2^{\ell} \\
            &\leq 2^{\ell} n^{\ell - 1 + \epsilon},
        \end{align*}
        as desired.
    \end{proof}

    \begin{claim}\label{bound-2}
        $|\mB_2| \leq n^{\ell - (\ell - 1)\gamma}$.
    \end{claim}

    \begin{proof}
        By the definition of the algorithm, we have $|W| \leq n^{1-\gamma}$ for every $W \in \mP$ (where $\mP$ denotes the partition at the end of the algorithm).
        Hence,
        \begin{align*}
            |\mB_2| &\leq \sum_{W \in \mathcal{P}}\binom{|W|}{\ell} \\
            &\leq (n^{1 - \gamma})^{\ell - 1} \cdot 
            \sum_{W \in \mathcal{P}}|W| \\
            &= n^{\ell - (\ell - 1)\gamma},
        \end{align*}
        as desired.
    \end{proof}

    We now complete the proof of the theorem. Consider the auxiliary $\ell$-graph on the vertex set $V(G)$ and edge set $\mB = \mB_1 \cup \mB_2$. Let us denote this hypergraph by $\mB$ as well. By Claims \ref{bound-1} and \ref{bound-2}, 

    \begin{align*}
        |E(\mB)| &\leq 2^{\ell} n^{\ell - 1 + \epsilon} + n^{\ell - (\ell - 1)\gamma} \\
        &= 2^{\ell} n^{\ell - 1 + \frac{\ell + 1}{3\ell - 1}} + n^{\ell - (\ell - 1) \cdot \frac{2}{3\ell - 1}} \\
        &= 2^{\ell} n^{\frac{3\ell^2 - 3\ell - \ell + 1 + \ell + 1}{3\ell - 1}} + n^{\frac{3\ell^2 - \ell - 2\ell + 2}{3\ell - 1}} \\
        &= (2^{\ell} + 1)n^{\frac{3\ell^2 - 3\ell + 2}{3\ell - 1}}.
    \end{align*}
    Hence, the average degree of $\mB$ is at most

    $$O\big( n^{\frac{3\ell^2 - 3\ell + 2}{3\ell - 1}-1} \big) = O\big(n^{\frac{3\ell^2 - 6\ell + 3}{3\ell - 1}}\big) = O\big(n^{\frac{3(\ell - 1)^2}{3\ell - 1}}\big).$$
    Thus, by Lemma \ref{spencer-ind-set-bound}, 
    

    \begin{align*}
        \alpha(\mB) &\geq \Omega\Big( 
        n^{1 - \frac{3(\ell - 1)^2}{3\ell - 1} \cdot \frac{1}{\ell - 1}} \Big) \\
        &= \Omega(n^{1 - \frac{3\ell - 3}{3\ell - 1}}) \\
        &= \Omega(n^{\frac{2}{3\ell - 1}}).
    \end{align*}

    Thus, by Claim \ref{extracted-claim}, $G$ contains an induced cohypergraph of size $\Omega(n^{\frac{2}{3\ell - 1}})$. By Lemma \ref{homogeneous-bound-1},

    \begin{align*}
        \max( \alpha(G), \omega(G) ) &\geq 
        \Omega(n^{\frac{1}{3\ell - 1}}),
    \end{align*}
    as desired.
    %
\end{proof}

Note that the above proof in fact shows that every $n$-vertex disperse $\ell$-graph can be made into a cohypergraph by adding/deleting 
$O(n^{\frac{3\ell^2-3\ell+2}{3\ell-1}}) = O(n^{\ell - \frac{2\ell-2}{3\ell-1}}) = o(n^{\ell})$ edges. Indeed, changing the $\ell$-tuples in $\mathcal{B}$ in an appropriate way yields a cohypergraph.

\section{Concluding remarks}\label{sec:concluding}
For a set $L \subseteq \{0,1,\dots,\ell+1\}$, let us call an $\ell$-graph {\em $L$-free} if it contains no induced $(\ell+1)$-vertex subgraph whose number of edges belongs to $L$. Thus, an $\ell$-graph is disperse if and only if it is $\{2,3,\dots,\ell-1\}$-free. Our main result is that disperse $\ell$-graphs have polynomial-size homogeneous sets. 
\begin{problem}
Is there any proper subset $L \subsetneq \{2,3,\dots,\ell-1\}$ for which $L$-free $\ell$-graphs have polynomial-size homogeneous sets? 
\end{problem}
\noindent
This trivially fails for $\ell=3$, and can also be shown to fail for $\ell=4$; but such an $L$ could exist for $\ell \geq 5$. 

By combining Lemmas \ref{lem:links are disperse} and \ref{gishboliner-and-tomon-result}, we get that in a disperse $\ell$-graph $G$, for every $\ell-2$ distinct vertices $v_1,\dots,v_{\ell-2}$, the link graph $L(v_1,\dots,v_{\ell-2})$ is a cograph (this graph consists of all pairs $xy$ such that $\{v_1,\dots,v_{\ell-2},x,y\} \in E(G)$). Thus, the following would be a strengthening of Theorem \ref{thm:EH}.

\begin{problem}
    Is there $c_{\ell} > 0$ such that if $G$ is an $n$-vertex $\ell$-graph in which all links $L(v_1,\dots,v_{\ell-2})$ are cographs, then $G$ has a homogeneous set of size at least $n^{c_{\ell}}$?
\end{problem}

Another well-studied graph class is that of split graphs. A graph is {\em split} if it has a vertex partition $X,Y$ such that $X$ is a clique and $Y$ is an independent set. One can ask for the size of homogeneous sets in $3$-graphs in which every link is split. It turns out that such 3-graphs might only have homogeneous sets of logarithmic size, due to the following variant of a well-known construction. Take a random graph $F \sim G(n,\frac{1}{2})$, and define a 3-graph $G$ on $V(F)$ in which $xyz$ is an edge if $e_F(\{x,y,z\}) \geq 2$. It is easy to show, using standard arguments, that $\omega(G),\alpha(G) = O(\log n)$. Also, for every vertex $v$, $N_F(v)$ is a clique in the link of $v$, and $V(F) \setminus (N_F(v) \cup \{v\})$ is an independent set in the link of $v$. Hence each link is split.

We can also show that the above construction is inevitable, in the sense that if $G$ is a 3-graph where all links are split, then $G$ contains a large vertex-set which induces the above construction.
\begin{proposition}\label{lem:split graph links}
		Let $G$ be an $n$-vertex 3-graph in which every link is split. Then there is $U \subseteq V(G)$, $|U| \geq n^{1/3}$, and a graph $F$ on $U$, such that for every distinct $x,y,z \in U$, $xyz \in E(G)$ if and only if $e_F(\{x,y,z\}) \geq 2$. 
\end{proposition}
	\begin{proof}
		By assumption, for each $v \in V(G)$, there is a partition $V(G) \setminus \{v\} = A_v \cup B_v$ such that $A_v$ is independent in the link of $v$ and $B_v$ is a clique in the link of $v$. Call a pair of vertices $xy$ {\em good} if $x \in A_y \Longleftrightarrow y \in A_x$ (i.e., $x \in A_y$ and $y \in A_x$, or $x \in B_y$ and $y \in B_x$). Otherwise, call $xy$ {\em bad}. We first claim that there is no $K_4$ consisting of bad pairs. Suppose otherwise; let $x_1,x_2,x_3,x_4$ such that $x_ix_j$ is bad for every $1 \leq i < j \leq 4$. Orient the edges $x_ix_j$ by letting $x_i \rightarrow x_j$ if $x_j \in A_{x_i}$ (note that exactly one of $x_j \in A_{x_i}$, $x_i \in A_{x_j}$ holds). Every tournament on 4 vertices has a transitive tournament on 3 vertices. Hence, without loss of generality, we can assume that $x_1 \rightarrow x_2,x_3$ and $x_2 \rightarrow x_3$. It follows that $x_2,x_3 \in A_{x_1}$, and hence $x_1x_2x_3 \notin E(G)$. On the other hand, $x_1,x_2 \in B_{x_3}$, and hence $x_1x_2x_3 \in E(G)$, giving a contradiction.
		
		We showed that the graph of bad pairs has no $K_4$. As $R(K_4,K_m) \leq \binom{m+2}{3} \leq m^3$, there is a set $U \subseteq V(G)$, $|U| \geq n^{1/3}$, such that $U$ contains no bad pairs. Now define a graph $F$ on $U$ by letting $xy$ be an edge of $F$ if $x \in B_y$ and $y \in B_x$ (so $xy$ is not an edge of $F$ if $x \in A_y$ and $y \in A_x$). Let $x,y,z \in U$ be distinct. If $xy,xz \in E(F)$, then $y,z \in B_x$, so $xyz \in E(G)$. Similarly, if $xy,xz \notin E(F)$, then $y,z \in A_x$ and so $xyz \notin E(F)$. This proves the proposition.  
	\end{proof}

\bibliographystyle{plain} 
\bibliography{library}

\end{document}